\documentclass[12pt]{amsart}
\usepackage{amsopn}
\usepackage{amssymb}
\usepackage{graphicx}
\usepackage[all]{xy}
\usepackage{amscd}
\usepackage{color}
\usepackage[colorlinks]{hyperref}
 \newtheorem{thm}{Theorem}[section]
 \newtheorem{cor}[thm]{Corollary}
 \newtheorem{lem}[thm]{Lemma}
 \newtheorem{prop}[thm]{Proposition}
 
 \newtheorem{defn}[thm]{Definition}

 \theoremstyle{definition}
 \theoremstyle{remark}
\allowdisplaybreaks
\begin{document}

\title[OPERATOR STRUCTURE OF FRAMES]
{On abstract results of Operator representation of frames in
Hilbert spaces}

\author[J. Cheshmavar]{Jahangir Cheshmavar$^{*}$}

\address{Department of Mathematics, Payame Noor University, P.O.BOX 19395-3697, Tehran, IRAN.}
\email{j$_{_-}$cheshmavar@pnu.ac.ir}

\author[A. Dallaki]{Ayyaneh Dallaki }
\address{Department of Mathematics, Payame Noor University, P.O.Box 19395-3697, Tehran, Iran.}
\email{ayyanehdallaki@student.pnu.ac.ir }

\author[J. Baradaran]{ Javad Baradaran}

\address{Department of Mathematics, Jahrom University, P.B.7413188941, Jahrom, Iran.}
\email{baradaran@jahromu.ac.ir}


\keywords{Frame; iterated action; operator representation; Riesz
representation theorem; Hahn-Banach theorem; spectrum of an
operator.
\\
\indent $^{*}$ Corresponding author} \maketitle

\begin{abstract}
In this paper, we give a multiplication operator representation of
bounded self-adjoint operators $T$ on a Hilbert space
$\mathcal{H}$ such that $\{T^k\varphi\}_{k=0}^{\infty}$ is a frame
for $\mathcal{H}$, for some $\varphi \in \mathcal{H}$. We state a
necessary condition in order for a frame $\{f_k\}_{k=1}^{\infty}$
to have a representation of the form $\{T^kf_1\}_{k=0}^{\infty}$.
Frame sequence $\{T^k\varphi\}_{k=0}^{\infty}$ with the synthesis
operator $U$, have also been characterized in terms of the
behavior spectrum of $U^*U|_{N(U)^{\perp}}$. Next, using the
operator response of an element with respect to a unit vector in
$\mathcal{H}$, frames $\{f_k\}_{k=1}^{\infty}$ of the form
$\{T^nf_{1}\}_{n=0}^{\infty}$ are characterized. We also consider
stability frames as $\{T^k\varphi\}_{k=0}^{\infty}$. Finally, we
conclude this note by raising a conjecture connecting frame theory
and operator theory.
\end{abstract}

% ----------------------------------------------------------------

\section{\textbf{Introduction and Preliminaries}}\label{Sec1}
The representation problem of a  frame in a Hilbert space
$\mathcal{H}$ as the iterative sequence
$\{T^k\varphi\}_{k=0}^{\infty}$, for some linear operator $T$ on
$\mathcal{H}$ and for some $\varphi\in \mathcal{H}$ is called
dynamical sampling problem.  This subject is a new research topic
in harmonic analysis and it has been studied by Aldroubi and
Petrosyan \cite{Aldroubi1.2016} to give  results on frames
obtaining via the iterated actions of normal operators on finite
dimensional spaces (e.g. see \cite{Aldroubi2.2017,
Aldroubi3.2017}). Recently, Christensen et al. wrote two papers
\cite{Christensen.2017, Christensen.2019} on the operator
representation of frames and  frame properties arising via the
iterated actions of linear  operators on a Hilbert spaces.

 We refer the reader to \cite{Aldroubi1.2016, Aldroubi3.2017, Bayart, Christensen.2016, Philipp} for an introduction of frame theory, dynamical sampling problem, and its
applications. Let $\{f_k\}_{k=1}^{\infty}$ be a frame for a
Hilbert space $\mathcal{H}$ which spans an infinite dimensional
subspace of $\mathcal{H}$. A natural question to ask
 whether is there a linear operator $T$ on $\mathcal{H}$ such that
$f_{k+1}=Tf_k$, for all $k\in \mathbb{N}$. It is proved in
\cite{Christensen.2019}  that such an operator exists if and only
if the frame $\{f_k\}_{k=1}^{\infty}$ is linearly independent
(Proposition \ref{prop.1}). Also,  it is shown that the operator
$T$ is bounded if and only if the kernel of the synthesis operator
of $\{f_k\}_{k=1}^{\infty}$ is invariant under the right shift
operator on $\ell^2(\mathbb{N}).$ In the affirmative case;
$\{f_k\}_{k=1}^{\infty}=\{T^kf_1\}_{k=0}^{\infty}$ (see
\cite[Theorem (2.3)]{Christensen.2017}).

In this work, we study frames of the from
$\{T^k\varphi\}_{k=0}^{\infty}$ in Hilbert spaces. Applying
functional calculus whenever $T$ is a bounded self-adjoint
operator on $\mathcal{H}$, we give a multiplication operator
representation of $T$ such that for some $\varphi \in
\mathcal{H}$, the sequence $\{T^k\varphi\}_{k=0}^{\infty}$ is a
frame for $\mathcal{H}$ (Proposition \ref{prop.3}). Furthermore,
necessary condition under which a frame $\{f_k\}_{k=1}^{\infty}$
in a Hilbert space $\mathcal{H}$ to has a representation of the
form $\{T^kf_1\}_{k=0}^{\infty}$ is obtained (Proposition
\ref{prop.4}). We give a characterization of the frame sequence
$\{T^k\varphi\}_{k=0}^{\infty}$ with the synthesis operator $U$ in
terms of the behavior spectrum of $U^*U|_{N(U)^{\perp}}$
(Proposition \ref{Prop.5}). Using the operator response of an
element with respect to a unit vector in $\mathcal{H}$, we
characterize frames $\{f_k\}_{k=1}^{\infty}$ of the form
$\{T^kf_1\}_{k=0}^{\infty}$, with bounded linear operator $T$
(Proposition \ref{prop.6}). At the end, we consider stability
frames of the forms
 $\{T^k\varphi\}_{k=0}^{\infty}$ (Proposition \ref{prop.8}).

Throughout this work,  $\mathcal{H}$ denotes a separable Hilbert
space. As usual, the set of all bounded linear operators on
$\mathcal{H}$ is denoted by $B(\mathcal{H})$. The index set is the
natural numbers $\mathbb{N}$ and $\mathbb{N}_0:=\mathbb{N} \cup
\{0\}$. The linear span of a set $\mathcal{V}$ will denote by
$span\mathcal{V}$, and whose closure denotes by
$\overline{span}\mathcal{V}$. Finally, the notation $\sigma(T)$
for the spectrum of an operator $T \in B(\mathcal{H})$ is used.

\begin{defn}\cite{Christensen.2016}
A sequence of vectors $\{f_k\}_{k=1}^{\infty} \subset \mathcal{H}$
is a \textit{frame} for $\mathcal{H}$ if there exist constants
$A,B>0$ such that
\begin{eqnarray}\label{form1}
A \parallel f\parallel^2\leq \sum_{k=1}^{\infty}\mid\langle f,
f_k\rangle \mid^2\leq B
\parallel f\parallel^2, \; \forall f \in \mathcal{H}.
\end{eqnarray}

\end{defn}
The sequence $\{f_k\}_{k=1}^{\infty}$ is called a \textit{frame
sequence} for $\mathcal{H}$ if the inequality (\ref{form1}) holds
for all $f\in \overline{span}\{f_k\}_{k=1}^{\infty}$. Also,
$\{f_k\}_{k=1}^{\infty}$ is said to be a \textit{Bessel sequence}
for $\mathcal{H}$ whenever the upper condition in (\ref{form1})
holds. If $A=B$, it is said to be an $A$-tight frame for
$\mathcal{H}$. It follows from definition that if
$\{f_k\}_{k=1}^{\infty}$ is a frame for $\mathcal{H}$, then we
have $\overline{span}\{f_k\}_{k=1}^{\infty}=\mathcal{H}.$ If
$\{f_k\}_{k=1}^{\infty}$ is a Bessel sequence for $\mathcal{H}$,
then the \textit{synthesis operator} is given by
\begin{eqnarray*}
U:\ell^2(\mathbb{N})\rightarrow \mathcal{H};\,\
U(\{c_k\}_{k=1}^{\infty})=\sum_{k=1}^{\infty}c_k f_k~.
\end{eqnarray*}
It is known that $U$ is a well-defined and bounded operator (see
\cite{Christensen.2016}). The adjoint operator $U^{*}$ is given by
$U^{*}f=\{\langle f, f_{k} \rangle\}_{k=1}^{\infty}$, for all $f
\in \mathcal{H}$ and the frame operator is defined by
\begin{eqnarray*}
S:=UU^{*}:\mathcal{H} \rightarrow \mathcal{H};\,\
Sf=\sum_{k=1}^{\infty} \langle f,f_k\rangle f_k~.
\end{eqnarray*}

The availability of the representation
{$\{f_k\}_{k=1}^{\infty}=\{T^kf_1\}_{k=0}^{\infty}$ is
characterized in \cite{Christensen.2019}:

\begin{prop}\label{prop.1}
Consider any sequence $\{f_k\}_{k=1}^{\infty}$ in $\mathcal{H}$
for which $span\{f_k\}_{k=1}^{\infty}$ is infinite-dimensional.
Then the following are equivalent:
\begin{itemize}
\item[(i)] $\{f_k\}_{k=1}^{\infty}$ is linearly independent.
\item[(ii)] There exists a linear operator
$T:span\{f_k\}_{k=1}^{\infty} \rightarrow \mathcal{H}$ such that
$\{f_k\}_{k=1}^{\infty}=\{T^kf_1\}_{k=0}^{^{\infty}}$~.
\end{itemize}
\end{prop}

We need the following results in the sequel.

\begin{prop}\cite{Christensen.2019}
\label{prop.2} If $\{T^k\varphi\}_{k=0}^{\infty}$ is a frame for
some operator $T \in B(\mathcal{H})$ and some $\varphi \in
\mathcal{H}$, then $T$ has closed range.
\end{prop}

\begin{thm}\cite{Aldroubi1.2016}
\label{thm.1} Let $T \in B(\mathcal{H})$ and $f \in \mathcal{H}$
such that $\{T^n f\}_{n=0}^{\infty}$ is a frame for $\mathcal{H}$.
Then for every $\varphi \in \mathcal{H}, (T^{\ast})^n\varphi
\rightarrow 0$ as $n\rightarrow \infty$.
\end{thm}

\begin{thm}\cite{Cazassa.1997}
\label{thm.2} Let $\{f_k\}_{k=1}^{\infty}$ be a frame for
$\mathcal{H}$ with bounds $A,B$. Let
$\{g_k\}_{k=1}^{\infty}\subseteq \mathcal{H}$ and assume that
there exist constants $\lambda_1, \lambda_2, \mu\geq 0$ such that
$\max(\lambda_1+\frac{\mu}{\sqrt{A}}, \lambda_2)<1$ and

\begin{eqnarray*}
\|\sum_{k=1}^nc_k( f_k-g_k)\|\leq \lambda_1 \|\sum_{k=1}^nc_k
f_k\|+ \lambda_2 \|\sum_{k=1}^nc_k g_k\|+\mu
(\sum_{k=1}^n|c_k|^2)^{1/2},
\end{eqnarray*}
for all $c_1,\cdots, c_n (n\in \mathbb{N})$. Then
$\{g_k\}_{k=1}^{\infty}$ is a frame with bounds
$$A(1-\frac{\lambda_1+\lambda_2+\frac{\mu}{\sqrt{A}}}{1+\lambda_2})^2,
\,\
B(1+\frac{\lambda_1+\lambda_2+\frac{\mu}{\sqrt{B}}}{1-\lambda_2})^2.$$
\end{thm}

\bigskip

%==============================================================
\section{\textbf{The results}}
For a given operator $T\in B(\mathcal{H})$, we define

\begin{eqnarray*}
\mathcal{V}(T):=\left\{\varphi \in \mathcal{H}: \{T^k\varphi
\}_{k=0}^{\infty}\; \mbox{is a frame for}\; \mathcal{H} \right\},
\end{eqnarray*}
and
\begin{eqnarray*}
E(\mathcal{H}):=\left\{T \in B(\mathcal{H}):\; \{T^k\varphi
\}_{k=0}^{\infty}\; \mbox{is a frame for}\; \mathcal{H},\;
\mbox{for some}\; \varphi \in \mathcal{H} \right\}.
\end{eqnarray*}
\smallskip

The following proposition is a multiplication operator
representation of self-adjoint operators in $E(\mathcal{H})$.
Recall that $C(X)$ denotes the set of all continuous functions on
a locally compact space $X$ and $C_{0}(X)$ contains all functions
that vanish at infinity. If $X$ is compact, then $C_{0}(X)=C(X).$

\begin{prop}
\label{prop.3} Let $T\in E(\mathcal{H})$ be self-adjoint and the
sequence $\{T^{k}\varphi\}_{k=0}^{\infty}$ be a frame in
$\mathcal{H}$ for some $\varphi \in \mathcal{H}$. Then there
exists a unitary operator $V$ from $\mathcal{H}$ to
$L^{2}(\sigma(T), \mu_{\varphi})$ such that
$$(VTV^{*}f)(x)=xf(x),\; \mbox{for all} \; f \in L^{2}(\sigma(T),
\mu_{\varphi}).$$
\end{prop}
\begin{proof}
It is easy to see that the linear functional defined by
$$\Lambda : C_{0}(\sigma(T))\longrightarrow \mathbb{C},\; \; f {\longmapsto} \Lambda(f)=\langle \varphi, f(T)\varphi \rangle,$$ is continuous, so by Riesz representation theorem there exists a positive measure $\mu_{\varphi}$ with
$\mu_{\varphi}(\sigma(T))=||\varphi||^{2}$ such that for any $f
\in C_{0}(\sigma(T))$, we have
$$\Lambda(f)=\int_{\sigma(T)} f(\lambda)d\mu_{\varphi}(\lambda).$$
Now, for each $f \in C_{0}(\sigma(T))$, we define
$V(f(T)\varphi):=f$. Clearly, $V$ is well-defined and we obtain
\begin{eqnarray*}
||f(T)\varphi||^{2}&=&\langle \varphi, f(T)^{*}f(T)\varphi \rangle=\langle \varphi, (\overline{f}f)(T)\varphi \rangle \\
&=&\int_{\sigma(T)}|f(\lambda)|^{2}d\mu_{\varphi}(\lambda)=||f||^{2}.
\end{eqnarray*}
This shows that $V$ is an isometry on $C_{0}(\sigma(T))$, so it is
continuous. Hence, we can extend $V$ to the closure of
$\{f(T)\varphi ; f \in C_{0}(\sigma(T))\}$ which is equal to
$\mathcal{H}$; because on the one hand
\begin{eqnarray*}
\{P(T)\varphi: \; P \; is \; a \; polynomial\}=span\{T^{k}\varphi,
k=0, 1,..\},
\end{eqnarray*}
is dense in $\mathcal{H}$ and on the other hand, we have
$$\{P(T)\varphi: \; P \; is \; a \; polynomial\}\subseteq
\{f(T)\varphi : f \in C_{0}(\sigma(T))\}.$$ Moreover, the
extension is an isometry as well. Since $V$ is isometry, $ranV$ is
closed and  contains $C_{0}(\sigma(T))$.  But $C_{0}(\sigma(T))$
is dense in $L^{2}(\sigma(T), \mu_{\varphi})$, so it follows that
$ranV=L^{2}(\sigma(T), \mu_{\varphi})$. Therefore, $V$ is unitary
and $V^{*}f=f(T)\varphi$, for all $f \in L^{2}(\sigma(T),
\mu_{\varphi})$. Now, using the definition of $V$ and functional
calculus for bounded self-adjoint operators, for any $f \in
L^{2}(\sigma(T), \mu_{\varphi})$, we get
$$(VTV^{*}f)(x)=(VTf(T)\varphi)(x)=xf(x),$$
which completes the proof.
\end{proof}

%====================================================================

\smallskip
By Proposition \ref{prop.2}, if $\{f_k\}_{k=1}^{\infty}$ has a
representation of the form $\{T^k \varphi\}_{k=0}^{\infty}$, for
some $T \in B(\mathcal{H})$ and some $\varphi \in \mathcal{H}$,
then $T$ has closed range. The following result provides a
necessary condition under which a frame $\{f_k\}_{k=1}^{\infty}$
to has a structure form $\{T^k f_1\}_{k=0}^{\infty}$:

\begin{prop}
\label{prop.4} Let $\{f_k\}_{k=1}^{\infty}$ be a frame for
$\mathcal{H}$ of the form $\{T^kf_1\}_{k=0}^{\infty}$, where $T
\in B(\mathcal{H})$. Then for all $\lambda \in \mathbb{C}$, the
range of $T-\lambda I$ is dense in $\mathcal{H}$~.
\end{prop}

%\bigskip
\begin{proof}

Since $\{T^kf_1\}_{k=0}^{\infty}$ is a frame for $\mathcal{H}$, $T
\in E(\mathcal{H})$. If for $\lambda \in \mathbb{C}$ the range of
$T-\lambda I$ is not dense in $\mathcal{H}$, then pick $\varphi_0
\in \mathcal{H} \backslash \overline{(T-\lambda I)\mathcal{H}}$
such that
$\overline{span}\{T^k\varphi_0\}_{k=0}^{\infty}=\mathcal{H}$.
Therefore, there exists a $c>0$ such that
\begin{eqnarray*}
\|\phi - \phi_0\|>c,\,\ \mbox{for all}\,\ \varphi \in (T-\lambda
I)\mathcal{H}~.
\end{eqnarray*}
Let $\mathcal{M}$ be the subspace generated by $(T-\lambda
I)\mathcal{H}$ and $\varphi_0$. We define
\begin{eqnarray*}
\Lambda : \mathcal{M}\longrightarrow \mathbb{C}, \,\
\Lambda(\varphi+\alpha\varphi_0)=\alpha~,
\end{eqnarray*}
for any $\varphi \in (T-\lambda I)\mathcal{H}$ and for all $\alpha
\in \mathbb{C}$. Then since $$c|\alpha|< |\alpha|
\|\varphi_0+\alpha^{-1}\varphi\|=\|\alpha\varphi_0+\varphi\|~,$$
it is easy to see that $\Lambda$ is a continuous linear functional
on $\mathcal{M}$ and its norm is at most $c^{-1}$. Also,
$\Lambda=0$ on $(T-\lambda I)\mathcal{H}$ and
$\Lambda(\varphi_0)=1$. Hence by Hahn-Banach theorem it can be
extend to $\mathcal{H}$, i.e., $\Lambda$ is a continuous linear
functional on $\mathcal{H}$ with $\Lambda((T-\lambda
I)\mathcal{H})=\{0\}$ and $\Lambda(\varphi_0)\neq 0$. Therefore,
we obtain
\begin{eqnarray*}
\Lambda(T\varphi)=\lambda\Lambda(\varphi), \; \mbox{for all}\,\
\varphi \in \mathcal{H}~,
\end{eqnarray*}
 and then
 \begin{eqnarray*}
\Lambda(T^n\varphi)=\lambda^n\Lambda(\varphi),  \; \mbox{for all}
\; n \in \mathbb{N}~ \; \mbox{and} \;\varphi \in \mathcal{H}.
\end{eqnarray*}
In particular,
$\Lambda(T^n\varphi_0)=\lambda^n\Lambda(\varphi_0)$.\\

Because $0$, $\varphi \in \mathcal{H}$ and
$\overline{span}\{T^kf_1\}_{k=0}^{\infty}=\mathcal{H}$, hence
there exist two sequences of positive integers
$\{n_k\}_{k=1}^{\infty}$  and $\{m_k\}_{k=1}^{\infty}$  which tend
to $\infty$ as $k\longrightarrow \infty$ such that
\begin{eqnarray*}
\sum_{i=1}^{n_k}c_{n_i}T^{n_i}\varphi_0\longrightarrow 0,\,\
\mbox{as}\,\ k\longrightarrow \infty~,
\end{eqnarray*}
and
\begin{eqnarray*}
\sum_{i=1}^{m_k}d_{m_{i}}T^{m_{i}}\varphi_0\longrightarrow
\varphi_0, \,\ \mbox{as}\,\ k\longrightarrow \infty~.
\end{eqnarray*}
 Then the continuity of $\Lambda$  implies that
\begin{eqnarray*}
\sum_{i=1}^{n_k}c_{n_{i}} \lambda^{n_{i}}
\Lambda(\varphi_0)\longrightarrow 0, \,\ \mbox{as} \,\
k\longrightarrow \infty~,
\end{eqnarray*}
and
\begin{eqnarray*}
\sum_{i=1}^{m_k}d_{m_{i}}\lambda^{m_{i}}
\Lambda(\varphi_0)\longrightarrow \Lambda(\varphi_0),\,\ \mbox{as}
\,\ k\longrightarrow \infty~.
\end{eqnarray*}\\
Therefore, it follows that  $\Lambda(\varphi_0)= 0$, and this
contradiction proves the claim.
\end{proof}

We give a characterization of frame sequence
$\{T^k\varphi\}_{k=0}^{\infty}$  in terms of the behavior spectrum
of $U^*U|_{N(U)^{\perp}}$:

\begin{prop}\label{Prop.5}
Let $0\neq T \in B(\mathcal{H})$ and for some $\varphi \in
\mathcal{H}$, the series $\sum_{k=0}^{\infty}c_k T^k\varphi$
converges, for all $\{c_k\}_{k=0}^{\infty}\in
\ell^2(\mathbb{N}_0)$. Then $\{T^k\varphi\}_{k=0}^{\infty}$ is a
frame sequence for $\mathcal{H}$ if and only if the spectrum of
$U^*U|_{N(U)^{\perp}}$ is contained in the half-opened interval
$]0, \|U\|^2]$,
%$$\sigma(U^*U|_{N(U)^{\perp}})\subseteq (0, \|U\|^2],$$
where $(N(U))^{\perp}$ denotes the orthogonal complement of null
space $N(U)$. In this case, we have
$$\sigma(U^*U)\subseteq [0, \|U\|^2].$$
\end{prop}

\begin{proof} Because $\sum_{k=0}^{\infty}c_k T^k\varphi$ is well-defined
for all $\{c_k\}_{k=0}^{\infty}\in \ell^2(\mathbb{N}_0)$,
$\{T^k\varphi\}_{k=0}^{\infty}$ is a Bessel sequence for
$\mathcal{H}$ and whose the synthesis operator defined by
$$U(\{c_k\}_{k=0}^{\infty})= \sum_{k=0}^{\infty}c_k T^k\varphi,$$ is well-defined and bounded on
$\ell^2(\mathbb{N}_0)$. We first recall two well-known results for
a bounded linear operator $\Lambda : \mathcal{H}\longrightarrow
\mathcal{K}$, where $\mathcal{K}$ denotes a Hilbert space:
\begin{itemize}
\item $R(\Lambda)$ is closed in $\mathcal{K}$ if and only if
$R(\Lambda^*\Lambda)$ is closed in
$\mathcal{H}$, and \\
$$\overline{R(\Lambda^*)}=(N(\Lambda))^{\perp}=(N(\Lambda^*\Lambda))^{\perp}=\overline{R(\Lambda^*\Lambda)}.$$

\item $N(\Lambda)$ and $(N(\Lambda))^{\perp}$ are invariant under $\Lambda^*\Lambda,$ and\\
$$\sigma(\Lambda^*\Lambda)=\sigma(\Lambda^*\Lambda|_{N(\Lambda)}) \cup
\sigma(\Lambda^*\Lambda|_{N(\Lambda)^{\perp}})\subseteq [0,
\|\Lambda\|^2].$$
\end{itemize}
 We now consider positive operator: $$U^*U|_{N(U)^{\perp}}:(N(U))^{\perp}
\rightarrow (N(U))^{\perp}.$$ Clearly, it is injective and whose
range $R(U^*U|_{N(U)^{\perp}})=R(U^*U)$ is dense in
$(N(U))^{\perp}$. Now, we have
\begin{eqnarray*}
R(U)\,\ \mbox{is closed in}\,\ \mathcal{H}&\Longleftrightarrow&
R(U^*U) \,\ \mbox{is closed in}\,\ \mathcal{H},\\
&\Longleftrightarrow& U^*U|_{N(U)^{\perp}} \,\ \mbox{is
bijective},\\ &\Longleftrightarrow& 0 \notin
\sigma(U^*U|_{N(U)^{\perp}}),\\
&\Longleftrightarrow& \sigma(U^*U|_{N(U)^{\perp}})\subseteq ]0,
\|U\|^2].
\end{eqnarray*}
In this case, we obtain
\begin{eqnarray*}
\sigma(U^*U)=\sigma(U^*U|_{N(U)})\cup
\sigma(U^*U|_{N(U)^{\perp}})&\subseteq& \{0\}\cup ]0, \|U\|^2]\\
&=& [0, \|U\|^2].
\end{eqnarray*}
Now, using Corollary 5.5.2 in \cite{Christensen.2016} the proof is
complete.
\end{proof}

\bigskip
Let $e$ be a unit vector in $\mathcal{H}$. For every $g\in
\mathcal{H}$, we define $\Lambda_{g}^e:\mathcal{H} \rightarrow
\mathcal{H}$ by
\begin{eqnarray}
 \Lambda_{g}^ef=\langle f, g\rangle e, \quad\ \forall f \in
\mathcal{H}.
\end{eqnarray}
Then $\Lambda_{g}^e$ is a bounded linear operator on
$\mathcal{H}$ which is called \textit{operator response} of $g$ with respect to $e$.\\

Authors in \cite{Cabrelli.2020}, characterized frames of the form
$\{T^nf_k\}_{n\in \mathbb{N}, K\in I}$, where $T \in
B(\mathcal{H})$ is a normal operator, $I$ is a countable index
set, and $\{f_k\}_{k=1}^{\infty} \subset \mathcal{H}$. In the
following result, we characterize frames of the form
$\{T^nf_{1}\}_{n=0}^{\infty}$ with the help of operator response
of $f_1$ with respect to $e=\frac{f_1}{\|f_1\|}$:

\begin{prop}\label{prop.6}
Consider a Bessel sequence $\{f_k\}_{k=1}^{\infty}$ in
$\mathcal{H}$ having the form $\{T^nf_{1}\}_{n=0}^{\infty}$, for
some $T\in B(\mathcal{H})$. Then $\{f_k\}_{k=1}^{\infty}$ is a
frame for $\mathcal{H}$ if and only if the following conditions
satisfied:
\begin{itemize}
\item [(i)] $(T^{\ast})^n \varphi\rightarrow 0$ as $n\rightarrow
\infty$, for all $\varphi \in \mathcal{H}$.

\item [(ii)] There exists a boundedly invertible operator $S\in
B(\mathcal{H})$ such that $TST^{\ast}=S-\|f_1\|\Lambda_{f_1}^e$,
where $e=\frac{f_1}{\|f_1\|}$.
\end{itemize}
\end{prop}
\begin{proof}
Let $\{f_k\}_{k=1}^{\infty}=\{T^nf_{1}\}_{n=0}^{\infty}$ be a
frame for $\mathcal{H}$ with the frame operator $S$. Then by
Theorem \ref{thm.1}, (i) is satisfied. In order to prove (ii),
because $S$ is boundedly invertible and $e=\frac{f_1}{\|f_1\|}$,
for any $f\in \mathcal{H}$, we have
\begin{eqnarray}
TST^{\ast}f&=&\sum_{k=0}^{\infty}\langle T^{\ast}f,
T^kf_1\rangle T^{k+1}f_1 \notag \\
&=&\sum_{k=0}^{\infty}\langle f, T^{k+1}f_1\rangle T^{k+1}f_1 \notag \\
&=&\sum_{k=0}^{\infty}\langle f, T^kf_1\rangle T^kf_1-\langle f,
f_1\rangle f_1\\
&=& Sf-\|f_1\|\Lambda_{f_1}^ef \notag \\
&=&(S-\|f_1\|\Lambda_{f_1}^e)f, \notag
\end{eqnarray}
which proves (ii). Conversely, suppose that (i) and (ii) are
satisfied and  pick $\Lambda:=\|f_1\|\Lambda_{f_1}^e$. Then we
obtain
\begin{eqnarray*}
T^2S(T^{\ast})^2=T(TST^{\ast})T^{\ast}=T(S-\Lambda)T^{\ast}\\
=TST^{\ast}-T\Lambda T^{\ast}=S-(\Lambda+T\Lambda T^{\ast}).
\end{eqnarray*}
By induction, we get
\begin{eqnarray}\label{eq.6}
T^nS(T^{\ast})^n=S-\sum_{k=0}^{n-1}T^k\Lambda (T^{\ast})^k.
\end{eqnarray}

\bigskip

On the other hand, for any $f \in \mathcal{H}$ we get
\begin{eqnarray}\label{eq.7}
\Lambda
(T^{\ast})^kf=\|f_1\|T_{f_1}^e(T^{\ast})^kf&=&\|f_1\|\langle
(T^{\ast})^kf,f_1\rangle e \notag \\
&=&\langle (T^{\ast})^kf,f_1\rangle f_1.
\end{eqnarray}
Therefore, by (\ref{eq.6}) and (\ref{eq.7}) we have
\begin{eqnarray*}
T^nS(T^{\ast})^nf=Sf-\sum_{k=0}^{n-1}\langle f, T^kf_1\rangle
T^kf_1.
\end{eqnarray*}

Hence, for any $f \in \mathcal{H}$ we obtain

\begin{eqnarray*}
\langle S(T^{\ast})^nf, (T^{\ast})^nf \rangle =\langle Sf,f\rangle
- \sum_{k=0}^{n-1}|\langle f, T^kf_1\rangle|^2.
\end{eqnarray*}
Applying (i) for any $f \in \mathcal{H}$, we get
\begin{eqnarray*}
\sum_{k=0}^{\infty}|\langle f, T^kf_1\rangle|^2=\langle
Sf,f\rangle-\lim_{n\rightarrow \infty}\langle S(T^{\ast})^nf,
(T^{\ast})^nf \rangle = \langle Sf,f\rangle.
\end{eqnarray*}
Therefore, $S$ is self-adjoint and nonnegative operator. Now, if
$S_1$ denote the frame operator of the Bessel sequence
$\{f_k\}_{k=1}^{\infty}=\{T^nf_{1}\}_{n=0}^{\infty}$, then we have
\begin{eqnarray*}
\langle S_1f,f\rangle =\sum_{k=0}^{\infty}|\langle f,
T^kf_1\rangle|^2.
\end{eqnarray*}
Therefore, for any $f \in \mathcal{H},\,\ \langle
Sf,f\rangle=\langle S_1f,f\rangle$, which implies $S=S_1$. Now the
boundedly invertible of the operator $S$, concludes the proof.
\end{proof}

In the following lemma, we give a characterization of frames
$\{f_k\}_{k=1}^{\infty}$ in terms of the frame coefficients
$\{\langle f, f_k \rangle\}_{k=1}^{\infty}$:
\begin{lem}\label{eqn.88}
The sequence $\{f_k\}_{k=1}^{\infty}\subset \mathcal{H}$ is a
frame for $\mathcal{H}$ if and only if for any $0 \neq f \in
\mathcal{H}$, $\{\langle f, f_k \rangle\}_{k=1}^{\infty}$ is a
frame for $\mathbb{C}$.
\end{lem}
\begin{proof}
Let $0 \neq f \in \mathcal{H}$, and let $\{f_k\}_{k=1}^{\infty}$
be a frame for $\mathcal{H}$ with bounds $A,B$. For any $\alpha,
\beta \in \mathbb{C}$, we have $\langle \alpha,\beta
\rangle=\alpha \overline{\beta}$. Therefore, we have
\begin{eqnarray*}
\sum_{k=1}^{\infty} | \left \langle \langle f, f_k\rangle,g \right
\rangle |^2 =\sum_{k=1}^{\infty} |\left\langle f, f_k
\right\rangle \overline{g}|^2=|g|^2\sum_{k=1}^{\infty}
|\left\langle f, f_k \right\rangle |^2, \,\ \forall g \in
\mathbb{C}.
\end{eqnarray*}

Using lower and upper frame inequalities of the frame
$\{f_k\}_{k=1}^{\infty}$, we get

\begin{eqnarray*}
A|g|^2 \|f\|^2 \leq \sum_{k=1}^{\infty} |\left\langle \langle f,
f_k\rangle,g \right\rangle|^2 \leq B |g|^2 \|f\|^2, \,\ \forall g
\in \mathbb{C}.
\end{eqnarray*}
That is, $\{\langle f, f_k \rangle\}_{k=1}^{\infty}$ is a frame
for $\mathbb{C}$ with bounds $A\|f\|^2$ and $B\|f\|^2$.
Conversely, let for any $0 \neq f \in \mathcal{H}, \{\langle f,
f_k \rangle\}_{k=1}^{\infty}$ be a frame for $\mathbb{C}$ with
bounds $A$ and $B$. Then for any $0 \neq g \in \mathbb{C}$ we have
\begin{eqnarray*}
A|g|^2 \leq \sum_{k=1}^{\infty} |\left\langle \langle f,
f_k\rangle,g \right\rangle|^2 =\sum_{k=1}^{\infty} |\left\langle
f, f_k \right\rangle \overline{g}|^2 \leq B |g|^2.
\end{eqnarray*}
Therefore, we obtain
\begin{eqnarray}\label{eqn77}
A \leq \sum_{k=1}^{\infty} |\left\langle f, f_k \right\rangle|^2
\leq B, \,\ \forall f \in \mathcal{H}\setminus \{0\}.
\end{eqnarray}
If we replace $f$ by $\frac{f}{\|f\|}$ in (\ref{eqn77}), then for
any $f \in \mathcal{H}\setminus \{0\}$, we get
\begin{eqnarray*}
A \|f\|^2 \leq \sum_{k=1}^{\infty} |\left\langle f, f_k
\right\rangle|^2 \leq B \|f\|^2,
\end{eqnarray*}
that is, $\{f_k\}_{k=1}^{\infty}$ is a frame for $\mathcal{H}$.
\end{proof}

If we consider frames with indexing over the set of integer
numbers $\mathbb{Z}$ instead of $\mathbb{N}$, then we have the
following result.
\begin{prop}
If $\{T^k f_0\}_{k\in \mathbb{Z}}$ is a tight frame in
$\mathcal{H}$ for some invertible operator $T \in B(\mathcal{H})$,
then $T$ is an unitary.
\end{prop}
\begin{proof}
Since $\{T^kf_0\}_{k\in \mathbb{Z}}$ is a tight frame,
 using Theorem 2.3 in \cite{Christensen.2017}, for any $f \in \mathcal{H}$, we have\\
\begin{eqnarray}\label{eqn88}
\|f\|&=&\|T^{-1}Tf\|\leq \|T^{-1}\| \|Tf\| \\\notag
&=&\|Tf\|\leq\|T\|\|f\|=\|f\|.
\end{eqnarray}
Hence, $\|Tf\|=\|f\|$ for all $f \in \mathcal{H}$, so $T$ is an
isometry, i.e., $T^*T=I$. Similarly, using that $f=TT^{-1}f$ in
(\ref{eqn88}), we conclude that $T^{-1}$ is an isometry i.e,
$(T^{-1})^{*}T^{-1}=I$.

Therefore, it follows that $TT^*=T^{*}T=I$ and $T$ is a unitary.
\end{proof}
\begin{cor}
Consider a tight frame having a structure form $\{T^kf_0\}_{k\in
\mathbb{Z}}$, for some invertible operator $T \in B(\mathcal{H})$.
If $\{U^k g_0\}_{k\in \mathbb{Z}}$ is a dual frame of
$\{T^kf_0\}_{k\in \mathbb{Z}}$ for some $U \in B(\mathcal{H})$ and
for some $g_{0}\in \mathcal{H}$, then $U=T$.
\end{cor}
\begin{proof}
By the previous Proposition,  we have $T^{\ast}=T^{-1}$. On the
other hand, by the definition, for any $f \in \mathcal{H}$, we get
$$f=\sum_{k \in \mathbb{Z}}\langle f,\; U^{k}g_{0} \rangle T^{k}f_{0}=T\left(\sum_{k \in \mathbb{Z}}\langle U^{*}f,\; U^{k}g_{0} \rangle T^{k}f_{0} \right)=TU^{*}f,$$ i.e.  $TU^{*}=I.$
Since $T$ is invertible and $T^{\ast}=T^{-1}$, we obtain
$U^{*}=T^{*}$.
\end{proof}
\smallskip

The final result states a perturbation condition that preserves
operator representation structure of a frame.

%=======================================================

\begin{prop}\label{prop.8}
Assume that $\{f_k\}_{k=1}^{\infty}=\{T^kf_k\}_{k=0}^{\infty}$ is
a frame for $\mathcal{H}$ for some $T$ on $\mathcal{H}$ and
$\{g_k\}_{k=1}^{\infty}$ is a sequence in $\mathcal{H}$. If for
$0<\lambda_{1},\; \lambda_{2}<1$ and any $f \in \mathcal{H}$, we
have
\begin{eqnarray}\label{Prop.33}
\|\sum_{k=1}^{\infty}c_k\langle f_k-g_k,f\rangle\|\leq \lambda_{1}
\|\sum_{k=1}^{\infty}c_k\langle f_k,f\rangle\|+ \lambda_{2}
\|\sum_{k=1}^{\infty}c_k\langle g_k,f\rangle\|,
\end{eqnarray}
for all finite sequences $\{c_k\}_{k=1}^{\infty}$. Then
$\{g_k\}_{k=1}^{\infty}$ is a frame for $\mathcal{H}$ which has an
operator representation form.
\end{prop}

\begin{proof}
Take $h_k:=\langle f_k,\; f \rangle$ and $t_k:=\langle g_k,\;f
\rangle$ for all $f \in \mathcal{H}$ . Then the perturbation
condition (\ref{Prop.33}) implies that
\begin{eqnarray}\label{Prop.34}
\|\sum_{k=1}^{\infty}c_k (h_k-t_k)\|\leq \lambda_{1}
\|\sum_{k=1}^{\infty}c_k h_k\|+\lambda_{2}
\|\sum_{k=1}^{\infty}c_kt_k\|.
\end{eqnarray}
Since $\{f_k\}_{k=1}^{\infty}$ is a frame for $\mathcal{H}$, by
Lemma \ref{eqn.88}, for all $0\neq f \in {\mathcal{H}}$,
$\{h_k\}_{k=1}^{\infty}$ is a frame for $\mathbb{C}$. Now, by
Theorem \ref{thm.2}, $\{t_k\}_{k=1}^{\infty}$ is a frame for
$\mathbb{C}$. Again, by Lemma \ref{eqn.88},
$\{g_k\}_{k=1}^{\infty}$ is a frame for $\mathcal{H}$.

Suppose that for any finite subset $J\subset \mathbb{N}$,
$\sum_{k\in J}c_kg_k=0$. Applying inequality (\ref{Prop.33}), we
get
\begin{eqnarray*}
\|\sum_{k=1}^{\infty}c_k\langle f_k,f\rangle\|\leq \lambda_1
\|\sum_{k=1}^{\infty}c_k\langle f_k,f\rangle\|, \,\ \forall f \in
\mathcal{H}.
\end{eqnarray*}
Because $0<\lambda_1<1$, we obtain $\sum_{k\in J}c_k\langle
f_k,f\rangle=0$, i.e.,
\begin{eqnarray*}
\langle \sum_{k\in J}c_kf_k,f\rangle=0, \,\ \forall f \in
\mathcal{H}.
\end{eqnarray*}
Thus, $\sum_{k\in J}c_kf_k=0$. Now, by Proposition \ref{prop.1},
for any $k \in J$, we get $c_k=0$. Therefore,
$\{g_k\}_{k=1}^{\infty}$ is linearly independent and again
applying Proposition \ref{prop.1}, it has an operator
representation form $\{V^ng_k\}_{n=0}^{\infty}$ for some linear
operator $V$.
\end{proof}

%==========================================================

%\smallskip

%\begin{rem}
%It is clear that $E(\mathcal{H})$ cannot be dense in
%$B(\mathcal{H})$ with respect to the norm topology. In fact, by
%Proposition 2.2 in \cite{Christensen.2017}, every operator $T \in
%E(\mathcal{H})$ has the norm greater or equal than $1$. Hence, the
%norm topology is not always the most natural topology on
%$B(\mathcal{H})$. It is often  useful to consider the weakest
%topology on $B(\mathcal{H})$ so-called the strong operator
%topology (SOT). It is defined by the family of seminorms $\{p_h: h
%\in \mathcal{H}\}$, where $p_h(T)=\|Th\|$.
%\end{rem}

\smallskip

 We close this work by raising the following conjecture:

%\textbf{Conjecture 1}. Let $T \in E(\mathcal{H})$ be fixed and
%take $c=\|T\|+\alpha,\,\ (0<\alpha<1)$. Then the size of the set of all operators in
%$E(\mathcal{H})$ with the norm of at most $c$ is SOT-dense in the
%closed ball $\overline{B(0,c)}\subset B(\mathcal{H})$.

\smallskip

\textbf{Conjecture.} Let $T \in E(\mathcal{H})$. The Hilbert space
$\mathcal{H}$ can be decomposed into a direct sum of Hilbert
spaces $$\mathcal{H}=\bigoplus_{n \in \mathbb{N}}\mathcal{H}_n~,$$
where
\begin{itemize}
\item[(i)] $T\mathcal{H}_n \subset \mathcal{H}_n$. \item[(ii)] For
any $n \in \mathbb{N}$, there exists $\phi_n \in \mathcal{H}_n$
such that, $\{T^k\phi_n\}_{k=0}^{\infty}$ is a frame for
$\mathcal{H}_n$~.
\end{itemize}

%\bigskip

%\textbf{Author Contributions} The author's contributions are based
%on the order of the authors in the manuscript.

 %\smallskip

 %\textbf{Data availability} The manuscript has no associated data.
 %\smallskip

 %\textbf{Declarations}

 %\textbf{Conflicts of interest} The authors declare that they have no conflict of interest.
%\section{Acknowledgment}
%The authors thanks to ...........
%========================================================================

%\bibliography{aib}
\bibliographystyle{plain}
\end{document}